\documentclass[10pt,british,reqno]{amsart}
\usepackage[T1]{fontenc}
\usepackage[latin9]{inputenc}
\usepackage{geometry}
\usepackage{babel}
\usepackage{float}
\usepackage{amsthm}
\usepackage{amssymb}
\usepackage[unicode=true,
 bookmarks=false,
 breaklinks=false,pdfborder={0 0 1},backref=false,colorlinks=false]
 {hyperref}
\hypersetup{
 colorlinks,citecolor=red,urlcolor=green,linkcolor=blue}
\usepackage{breakurl}

\makeatletter
  \theoremstyle{plain}
  \newtheorem{prop}{\protect\propositionname}
  \theoremstyle{remark}
  \newtheorem{rem}{\protect\remarkname}
  \theoremstyle{plain}
  \newtheorem{cor}{\protect\corollaryname}
  \theoremstyle{plain}
  \newtheorem{lem}{\protect\lemmaname}
\theoremstyle{plain}
\newtheorem{thm}{\protect\theoremname}

\usepackage{url}

\makeatother

  \providecommand{\lemmaname}{Lemma}
  \providecommand{\propositionname}{Proposition}
  \providecommand{\remarkname}{Remark}
\providecommand{\corollaryname}{Corollary}
\providecommand{\theoremname}{Theorem}

\begin{document}

\title{The spectrum of diagonal perturbation of weighted shift operator}

\author{M. L. Sahari}

\address{M.L. Sahari: LANOS Laboratory, Department of Mathematics, Badji Mokhtar-Annaba
University, P. O. Box 12, 23000 Annaba, Algeria}

\email{mohamed-lamine.sahari@univ-annaba.dz; mlsahari@yahoo.fr}

\author{A. K. Taha}

\address{A.K. Taha: INSA, University of Toulouse, 135 Avenue de Rangueil,
31077 Toulouse Cedex 4, France }

\email{taha@insa-toulouse.fr}

\author{L. Randriamihamison}

\address{L. Randriamihamison: IPST-Cnam, Institut National Polytechnique de
Toulouse, University of Toulouse, 118, route de Narbonne, 31062 Toulouse
Cedex 9, France}

\email{louis.randriamihamison@ipst.fr}
\begin{abstract}
This paper provides a description of the spectrum of diagonal perturbation
of weighted shift operator acting on a separable Hilbert space.
\end{abstract}


\maketitle

\section{Introduction}

Let $X$ be a separable complex Hilbert space with an orthonormal
basis $\left\{ e_{i}\right\} _{i\in\mathbb{Z}}\subset X$. We define
the weighted shift operator in $X$ by 
\[
Se_{i}=w_{i}e_{i+1},\,\,\,i=0,\pm1,\pm2,...
\]
The sequence $\left\{ w_{i}\right\} _{i\in\mathbb{Z}}\subset\mathbb{C}$
represents the weights of the operator $S$. The matrix of such operator
can be written as
\[
\small S=\left(\begin{array}{ccccccc}
\ddots & \ddots & \vdots & \vdots & \vdots & \vdots & \vdots\\
\ddots & 0 & 0 & 0 & 0 & 0 & \cdots\\
\ddots & w_{-1} & 0 & 0 & 0 & 0 & \cdots\\
\cdots & 0 & w_{0} & 0 & 0 & 0 & \cdots\\
\cdots & 0 & 0 & w_{1} & 0 & 0 & \cdots\\
\cdots & 0 & 0 & 0 & w_{2} & 0 & \cdots\\
\vdots & \vdots & \vdots & \vdots & \ddots & \ddots & \ddots
\end{array}\right).
\]
In \cite{bourhim2006spectrum,conway2013course,shields1974weighted},
it is shown that if $S$ is bounded, then there exists $0\leq r^{-}\leq r^{+}$
such that the spectrum $\sigma(S)$ of $S$ is given by 
\[
\sigma(S)=\left\{ \lambda\in\mathbb{C}\,:\,r^{-}\leq\left|\lambda\right|\leq r^{+}\right\} .
\]
In this work, we propose to extend this type of result to the case
of the perturbed operator $S+D$, where $D$ is a diagonal operator.

\section{Preliminary notions}

Let $\mathcal{L}(X)$ denote the algebra of all bounded linear operators
acting on a complex Banach space $X$. The norm on $X$ and the associated
operator norm on $\mathcal{L}(X)$ are both denoted by $\left\Vert \cdot\right\Vert $.
For $T\in\mathcal{L}(X)$, we denote by $\sigma(T)$, $\rho(T)$ and
$r(T)$ the spectrum, the resolvent and the spectral radius of $T$
respectively. Recall that $\sigma(T)$ is a non-empty compact subset
of $\mathbb{C}$, $r(T)\leq\left\Vert T\right\Vert $ and $r(T)=\lim\left\Vert T^{k}\right\Vert ^{\frac{1}{k}}=\inf\left\Vert T^{k}\right\Vert ^{\frac{1}{k}}$.
If $T$ is invertible, the inverse is denoted by $T^{-1}$ and we
have $\sigma(T^{-1})=\left\{ \dfrac{1}{\lambda}\,:\,\lambda\in\sigma(T)\right\} $.
Moreover, $\dfrac{1}{r(T^{-1})}=\inf\left\{ \left|\lambda\right|\,:\,\lambda\in\sigma(T)\right\} $
(see \cite{conway2013course,dowson1978spectral,halmos2012hilbert,hislop2012introduction,kato2013perturbation}).

\section{Some properties of weighted shift}

In the following, $X$ is a separable complex Hilbert space and $\left\{ e_{i}\right\} _{i\in\mathbb{Z}}$
an orthonormal basis of $X$. Let $S$ be a weighted shift operator
on $X$ with a weight sequence $\left\{ w_{i}\right\} _{i\in\mathbb{Z}}$.
The boundedness of the operator $S$ is a consequence of the boundedness
of the weight sequence $\left\{ w_{i}\right\} _{i\in\mathbb{Z}}$.
However, we have a more general results 
\begin{prop}[\cite{conway2013course,halmos2012hilbert,shields1974weighted}]
\label{bornetude}The operator $S$ is bounded if and only if the
weight sequence $\left\{ w_{i}\right\} _{i\in\mathbb{Z}}$ is bounded.
In this case, 
\[
\left\Vert S^{k}\right\Vert =\sup_{i\in\mathbb{Z}}\left|\prod_{m=0}^{k-1}w_{i+m}\right|,\,\,\,k=1,2,\cdots
\]
\end{prop}
\begin{rem}
If $S$ is invertible, then the inverse $S^{-1}$ is given by 
\[
S^{-1}e_{i}=\dfrac{1}{w_{i-1}}e_{i-1}
\]
and in this case 
\[
\begin{aligned}\left\Vert S^{-k}\right\Vert  & =\sup_{i\in\mathbb{Z}}\left|\prod_{m=0}^{k-1}\dfrac{1}{w_{i-m}}\right|=\sup_{i\in\mathbb{Z}}\left|\prod_{m=0}^{k-1}\dfrac{1}{w_{i+m}}\right|=\left[\inf_{i\in\mathbb{Z}}\left|\prod_{m=0}^{k-1}w_{i+m}\right|\right]^{-1},\,\,\,k=1,2,\cdots\end{aligned}
.
\]
\end{rem}
\begin{prop}
The operator $S$ is invertible if and only if the sequence $\left\{ \dfrac{1}{w_{i}}\right\} _{i\in\mathbb{Z}}$
is bounded. 
\end{prop}

\begin{prop}[\cite{conway2013course,halmos2012hilbert,shields1974weighted}]
\label{prop_unit_equ}If $\{\lambda_{i}\}_{i\in\mathbb{Z}}$ are
complex numbers of modulus 1, then $S$ is unitary equivalent to the
weighted shift operator with weight sequence $\{\overline{\lambda}_{i+1}\lambda_{i}w_{i}\}.$
\end{prop}
\begin{cor}
\label{unit_equiv}The operator $S$ is unitary equivalent to the
weighted shift operator of weight $\left\{ \left|w_{i}\right|\right\} _{i\in\mathbb{Z}}$
.
\end{cor}

\begin{cor}[\cite{bourhim2006spectrum,halmos2012hilbert}]
If $\left|c\right|=1$, then $S$ and $cS$ are unitary equivalent.
\end{cor}
\begin{rem}
From the last corollary, the spectrum of the operator $S$ have circular
symmetry about the origin. In particular, $\left\{ \lambda\in\mathbb{C}\,:\,\left|\lambda\right|=r(S)\right\} \subset\sigma(S)\,\,$
and $\,\,\left\{ \lambda\in\mathbb{C}\,:\,\left|\lambda\right|=\dfrac{1}{r(S^{-1})}\right\} \subset\sigma(S^{-1})$. 

In the following section, we state our main result.
\end{rem}

\section{The spectrum of perturbed weighted shift}

Let $T\in\mathcal{L}(X)$ be a perturbed operator given by 
\begin{equation}
T=S+D,\label{6*}
\end{equation}
where $D$ is a diagonal operator with diagonals $\left\{ d_{i}\right\} _{i\in\mathbb{Z}}$. 

\begin{lem}
\label{lemmaI}If $T$ is invertible, then we have at least one of
the following two inequalities 
\begin{equation}
R_{S+D}^{+}=\lim_{k\rightarrow\infty}\left[\sup_{i\in\mathbb{Z}}\left|\dfrac{\prod\limits _{m=0}^{k-1}w_{i+m}}{\prod\limits _{m=0}^{k}d_{i+m}}\right|\right]^{\frac{1}{k}}\leq1\label{limite1}
\end{equation}
or 
\begin{equation}
R_{S+D}^{-}=\lim_{k\rightarrow\infty}\left[\sup_{i\in\mathbb{Z}}\left|\dfrac{\prod\limits _{m=1}^{k-1}d_{i-m}}{\prod\limits _{m=1}^{k}w_{i-m}}\right|\right]^{\frac{1}{k}}\leq1.\label{limite2}
\end{equation}
\end{lem}
\begin{proof}
Let $T$ be invertible and set $x_{i}=\sum_{j\in\mathbb{Z}}a_{j}^{i}e_{j}=T^{-1}e_{i}$
. Thus, we have 
\begin{equation}
\left\{ \begin{aligned}w_{j-1}a_{j-1}^{i}+d_{j}a_{j}^{i}=1, & \,\,\,\,\mbox{if }j=i,\\
\\
w_{j-1}a_{j-1}^{i}+d_{j}a_{j}^{i}=0, & \,\,\,\,\,\mbox{otherwise}.
\end{aligned}
\right.\label{*}
\end{equation}
and 
\begin{equation}
\begin{cases}
w_{i}a_{j}^{i+1}+d_{i}a_{j}^{i}=1, & \,\,\,\,\mbox{if }j=i,\\
\\
w_{i}a_{j}^{i+1}+d_{i}a_{j}^{i}=0, & \,\,\,\,\,\mbox{otherwise}.
\end{cases}\label{**}
\end{equation}
The first equation of (\ref{*}) and of (\ref{**}) implies for all
$i\in\mathbb{Z}$, 
\begin{equation}
a_{i}^{i}d_{i}=a_{0}^{0}d_{0}\label{diag}
\end{equation}
From (\ref{*}), we get, for all $i\in\mathbb{Z}$ and $k>0$, 
\[
a_{i-k}^{i}=\left\langle T^{-1}e_{i}\,,\,e_{i-k}\right\rangle =(-1)^{k+1}\dfrac{(1-d_{i}a_{i}^{i})\prod\limits _{m=1}^{k-1}d_{i-m}}{\prod\limits _{m=1}^{k}w_{i-m}},
\]
assuming that $\prod_{m=1}^{0}d_{i-m}=1$. Cauchy-Schwarz inequality
gives us 
\begin{equation}
\left|\dfrac{(1-d_{0}a_{0}^{0})\prod\limits _{m=1}^{k-1}d_{i-m}}{\prod\limits _{m=1}^{k}w_{i-m}}\right|\leq\left\Vert T^{-1}\right\Vert ,\,\,\,\mbox{for every \ensuremath{i\in\mathbb{Z}\:}and \ensuremath{k\geq0}. }\label{***}
\end{equation}
Consequently, for all $i\in\mathbb{Z}$ and $k>0$, we have 
\[
a_{i+k}^{i}=\left\langle T^{-1}e_{i}\,,\,e_{i+k}\right\rangle =(-1)^{k}\dfrac{d_{i}a_{i}^{i}\prod\limits _{m=0}^{k-1}w_{i+m}}{\prod\limits _{m=0}^{k}d_{i+m}},
\]
Cauchy-Schwarz inequality provides the inequality 
\begin{equation}
\left|d_{0}a_{0}^{0}\dfrac{\prod\limits _{m=0}^{k-1}w_{i+m}}{\prod\limits _{m=0}^{k}d_{i+m}}\right|\leq\left\Vert T^{-1}\right\Vert ,\,\,\,\mbox{for every \ensuremath{i\in\mathbb{Z\,}}and \ensuremath{k>0}. }\label{****}
\end{equation}
From the first equation of (\ref{*}), for all $i\in\mathbb{Z}$,
either $w_{i-1}a_{i-1}^{i}$ or $d_{i}a_{i}^{i}$ is not zero. Thus,
we can distinguish two cases:

\textbf{1$^{st}$ case:} $d_{i}a_{i}^{i}\neq0$, from (\ref{diag})
and by taking the supremum over $i$ in (\ref{***}), we get 
\begin{equation}
\sup_{i\in\mathbb{Z}}\left|\dfrac{\prod\limits _{m=0}^{k-1}w_{i+m}}{\prod\limits _{m=0}^{k}d_{i+m}}\right|<\infty,\,\,\,\mbox{for every }k>0.\label{inequation1}
\end{equation}

\textbf{2$^{nd}$ case:} $w_{i-1}a_{i-1}^{i}\neq0$, from (\ref{diag})
and by taking the supremum over $i$ in (\ref{****}) , we get 
\begin{equation}
\sup_{i\in\mathbb{Z}}\left|\dfrac{\prod\limits _{m=1}^{k-1}d_{i-m}}{\prod\limits _{m=1}^{k}w_{i-m}}\right|<\infty,\,\,\,\mbox{for every }k>0.\label{inequation2}
\end{equation}
We conclude, by taking the $k$th root and letting $k\longrightarrow\infty$
in (\ref{inequation1}) and (\ref{inequation2}).
\end{proof}

In the folow, we give a converse of the previous lemma.
\begin{lem}
\label{reciproque}If $R_{S+D}^{+}<1$ or $R_{S+D}^{-}<1$ then $T$
is invertible.
\end{lem}
\begin{proof}
Note that for all $k>0$ 
\begin{equation}
\left[\sup_{i\in\mathbb{Z}}\left|\dfrac{\prod\limits _{m=0}^{k}d_{i-m}}{\prod\limits _{m=0}^{k}w_{i-m}}\right|\right]^{-1}=\inf_{i\in\mathbb{Z}}\left|\dfrac{\prod\limits _{m=0}^{k}w_{i-m}}{\prod\limits _{m=0}^{k}d_{i-m}}\right|\leq\sup_{i\in\mathbb{Z}}\left|\dfrac{\prod\limits _{m=0}^{k}w_{i+m}}{\prod\limits _{m=0}^{k}d_{i+m}}\right|,\label{onlyone}
\end{equation}
then only one of inequality $R_{S+D}^{+}<1$ or $R_{S+D}^{-}<1$ can
be satisfied.

Let $R_{S+D}^{+}<1$ and let $F$ be a linear operator on $X$ to
$X$, defined by
\begin{equation}
F:=\sum_{l=0}^{\infty}F_{l},\label{operatorsum}
\end{equation}
such as, for all $l\in\mathbb{N}$, $F_{l}$ is an operator given
by 
\begin{equation}
F_{l}e_{i}=a_{i+l}^{i}e_{i+l},\,\,\,i=0,\pm1,\pm2,...\label{oprerateurdefine}
\end{equation}
and
\begin{equation}
a_{i+l}^{i}=(-1)^{l}\dfrac{\prod\limits _{m=0}^{l-1}w_{i+m}}{\prod\limits _{m=0}^{l}d_{i+m}}\label{element-1}
\end{equation}
with the assumptions that $\prod_{m=0}^{-1}w_{i+m}=1$. The condition
$R_{S+D}^{+}<1$ implies that the operator $F$ is well defined, $\left\Vert F\right\Vert <\infty$
and $\lim_{k\longrightarrow\infty}a_{i+k}^{i+1}=\lim_{k\longrightarrow\infty}a_{i+k}^{i}=0$.
From (\ref{operatorsum})-(\ref{element-1}), and for all $i\in\mathbb{Z}$,
we have
\begin{eqnarray*}
\left(F\circ T\right)e_{i} & = & w_{i}Fe_{i+1}+d_{i}Fe_{i},\\
 & = & d_{i}a_{i}^{i}e_{i}+\lim_{k\longrightarrow\infty}\left\{ \left[\sum_{l=0}^{k}\left(w_{i}a_{i+l+1}^{i+1}+d_{i}a_{i+l+1}^{i}\right)e_{i+l+1}\right]+w_{i}a_{i+k+2}^{i+1}e_{i+k+2}\right\} ,\\
 & = & d_{i}a_{i}^{i}e_{i}+\sum_{l=0}^{\infty}\left(w_{i}a_{i+l+1}^{i+1}+d_{i}a_{i+l+1}^{i}\right)e_{i+l+1}.
\end{eqnarray*}
Equation (\ref{element-1}), leads to $a_{i}^{i}=\dfrac{1}{d_{i}}$
and 
\begin{equation}
w_{i}a_{i+l+1}^{i+1}=(-1)^{l}\dfrac{\prod\limits _{m=0}^{l}w_{i+m}}{\prod\limits _{m=0}^{l}d_{i+m+1}}.\label{A**1-3-1-1}
\end{equation}
Also
\begin{equation}
d_{i}a_{i+l+1}^{i}=-(-1)^{l}\dfrac{\prod\limits _{m=0}^{l}w_{i+m}}{\prod\limits _{m=0}^{l}d_{i+m+1}},\label{A**2-3-1-1}
\end{equation}
hence 
\[
w_{i}a_{i+l+1}^{i+1}+d_{i}a_{i+l+1}^{i}=0,
\]
 which implies $\left(F\circ T\right)e_{i}=e_{i}$ . Moreover, note
that
\begin{eqnarray*}
\left(T\circ F\right)e_{i} & = & T\left(\sum_{l=0}^{\infty}F_{l}e_{i}\right),\\
 & = & a_{i}^{i}d_{i}e_{i}+\lim_{k\longrightarrow\infty}\left\{ \left[\sum_{l=0}^{k}\left(a_{i+l}^{i}w_{i+l}+a_{i+l+1}^{i}d_{i+l+1}\right)e_{i+l+1}\right]+a_{i+k+1}^{i}w_{i+k+1}e_{i+k+1}\right\} ,\\
 & = & a_{i}^{i}d_{i}e_{i}+\sum_{l=0}^{\infty}\left(a_{i+l}^{i}w_{i+l}+a_{i+l+1}^{i}d_{i+l+1}\right)e_{i+l+1}.
\end{eqnarray*}
From (\ref{element-1}), we have $a_{i}^{i}=\dfrac{1}{d_{i}}$, then
\begin{equation}
w_{i+l}a_{i+l}^{i}=(-1)^{l}\dfrac{\prod\limits _{m=0}^{l}w_{i+m}}{\prod\limits _{m=0}^{l}d_{i+m}},\label{A**1-1-2-1-1}
\end{equation}
and
\begin{equation}
d_{i+l+1}a_{i+l+1}^{i}=-(-1)^{l}\dfrac{\prod\limits _{m=0}^{l}w_{i+m}}{\prod\limits _{m=0}^{l}d_{i+m}}.\label{A**2-1-2-1-1}
\end{equation}
Then
\[
w_{i+l}a_{i+l}^{i}+d_{i+l+1}a_{i+l+1}^{i}=0.
\]
 Hence, $\left(T\circ F\right)e_{i}=e_{i}$, which lead to
\[
T\circ F=F\circ T=I,
\]
where $I$ denotes the identity operator.

If $R_{S+D}^{+}<1$, let $F$ be an operator on $X$ to $X$, defined
by
\begin{equation}
F:=\sum_{l=1}^{\infty}F_{-l},\label{operatorsome1}
\end{equation}
and for all $l\in\mathbb{N}-\left\{ 0\right\} $, $F_{-l}$ is an
operator given by 
\begin{equation}
F_{-l}e_{j}=a_{j-l}^{j}e_{j-l},\,\,\,j=0,\pm1,\pm2,...\label{operatordefine1}
\end{equation}
and
\begin{equation}
a_{i-k}^{i}=(-1)^{k+1}\dfrac{\prod\limits _{m=1}^{k-1}d_{i-m}}{\prod\limits _{m=0}^{k}w_{i-m}},\label{element-1-1-1}
\end{equation}
with assumptions that $\prod_{m=1}^{0}d_{i-m}=1$. Note that, the
condition $R_{S+D}^{-}<1$ implies that the operator $F$ is well
defined, $\left\Vert F\right\Vert <\infty$ and $\lim_{k\longrightarrow\infty}a_{i-k}^{i}=0$.
From (\ref{operatorsome1})-(\ref{element-1-1-1}), then for all $i\in\mathbb{Z}$,
\begin{eqnarray*}
\left(F\circ T\right)e_{i} & = & w_{i}Fe_{i+1}+d_{i}Fe_{i},\\
 & = & w_{i}a_{i}^{i+1}e_{i}+\lim_{k\longrightarrow\infty}\left\{ \left[\sum_{l=1}^{n}\left(w_{i}a_{i-l}^{i+1}+d_{i}a_{i-l}^{i}\right)e_{i-l}\right]+d_{i}a_{i-k-1}^{i}e_{i-k-1}\right\} ,\\
 & = & w_{i}a_{i}^{i+1}e_{i}+\sum_{l=1}^{\infty}\left(w_{i}a_{i-l}^{i+1}+d_{i}a_{i-l}^{i}\right)e_{i-l}.
\end{eqnarray*}
Formula (\ref{element-1-1-1}), leads to $a_{i}^{i+1}=\dfrac{1}{w_{i}}$
and 
\begin{equation}
w_{i}a_{i-l}^{i+1}=-(-1)^{l+1}w_{i}\dfrac{\prod\limits _{m=1}^{l}d_{i-m+1}}{\prod\limits _{m=1}^{l+1}w_{i-m+1}}=-(-1)^{l+1}\dfrac{\prod\limits _{m=1}^{l}d_{i-m+1}}{\prod\limits _{m=1}^{l}w_{i-m}}.\label{A**1-2}
\end{equation}
 Also
\begin{equation}
d_{i}a_{i-l}^{i}=(-1)^{l+1}d_{i}\dfrac{\prod\limits _{m=1}^{l-1}d_{i-m}}{\prod\limits _{m=1}^{l}w_{i-m}}=(-1)^{l+1}\dfrac{\prod\limits _{m=1}^{l-1}d_{i-m+1}}{\prod\limits _{m=1}^{l}w_{i-m}}.\label{A**2-2}
\end{equation}
Combining (\ref{A**1-2}) and (\ref{A**2-2}), we obtain $w_{i}a_{i-l}^{i+1}+d_{i}a_{i-l}^{i}=0$.
Therefore $\left(F\circ T\right)e_{i}=e_{i}$. Moreover, note that
\begin{eqnarray*}
\left(T\circ F\right)e_{i} & = & T\left(\sum_{l=1}^{\infty}F_{-l}e_{i}\right),\\
 & = & w_{i-1}a_{i-1}^{i}e_{i}+\lim_{k\longrightarrow\infty}\left\{ \left[\sum_{l=1}^{k}\left(a_{i-l-1}^{i}w_{i-l-1}+a_{i-l}^{i}d_{i-l}\right)e_{i-l}\right]+a_{i-k-1}^{i}d_{i-k-1}e_{i-k-1}\right\} ,\\
 & = & w_{i-1}a_{i-1}^{i}e_{i}+\sum_{l=1}^{\infty}\left(a_{i-l-1}^{i}w_{i-l-1}+a_{i-l}^{i}d_{i-l}\right)e_{i-l},
\end{eqnarray*}
using (\ref{element-1-1-1}), we obtain $a_{i-1}^{i}=\dfrac{1}{w_{i-1}}$
and
\begin{equation}
w_{i-l-1}a_{i-l-1}^{i}=-(-1)^{l+1}w_{i-l-1}\dfrac{\prod\limits _{m=1}^{l}d_{i-m}}{\prod\limits _{m=1}^{l+1}w_{i-m}}=-(-1)^{l+1}\dfrac{\prod\limits _{m=1}^{l}d_{i-m}}{\prod\limits _{m=1}^{l}w_{i-m}}.\label{A**1-1-1}
\end{equation}
Also
\begin{equation}
d_{i-l}a_{i-l}^{i}=(-1)^{l+1}\dfrac{\prod\limits _{m=1}^{l}d_{i-m}}{\prod\limits _{m=1}^{l}w_{i-m}},\label{A**2-1-1}
\end{equation}
which leads to 
\[
w_{i-l-1}a_{i-l-1}^{i}+d_{i-l}a_{i-l}^{i}=0,
\]
therefore $\left(T\circ F\right)e_{i}=e_{i}$ and we have
\[
T\circ F=F\circ T=I,
\]
then the claim is proved.
\end{proof}

\begin{thm}
\label{thmprincipal}Let $T\in\mathcal{L}(X)$ be the operator given
by (\ref{6*}) and for any $\lambda\in\mathbb{C}$, $R_{S+D}^{+}(\lambda)$
, $R_{S+D}^{-}(\lambda)$ are given by 
\begin{equation}
R_{S+D}^{+}(\lambda)=\lim_{k\rightarrow\infty}\left[\sup_{i\in\mathbb{Z}}\left|\dfrac{\prod\limits _{m=0}^{k-1}w_{i+m}}{\prod\limits _{m=0}^{k}(d_{i+m}-\lambda)}\right|\right]^{\frac{1}{k}}\label{Gamma}
\end{equation}
and 
\begin{equation}
R_{S+D}^{-}(\lambda)=\lim_{k\rightarrow\infty}\left[\sup_{i\in\mathbb{Z}}\left|\dfrac{\prod\limits _{m=0}^{k-1}(d_{i-m}-\lambda)}{\prod\limits _{m=0}^{k}w_{i-m}}\right|\right]^{\frac{1}{k}}.\label{Delta}
\end{equation}
\\
 (i) If $S$ is an invertible operator, then 
\begin{equation}
\sigma(T)=\left\{ \lambda\in\mathbb{C}\,:\,R_{S+D}^{+}(\lambda)\geq1\,\,\,\mbox{and}\,\,\,R_{S+D}^{-}(\lambda)\geq1\right\} ;\label{spectre_inversible}
\end{equation}
(ii) if $S$ is a non-invertible operator, then 
\begin{equation}
\sigma(T)=\left\{ \lambda\in\mathbb{C}\,:\,R_{S+D}^{+}(\lambda)\geq1\right\} ,\label{spectre_non_inversible}
\end{equation}
\end{thm}
\begin{proof}
Let $\lambda\in\rho(T)=\left\{ \lambda\in\mathbb{C}\,:\,\left(T-\lambda I\right)^{-1}\in\mathcal{L}(X)\right\} $.
If we replace $d_{j}$ by $d_{j}-\lambda$ in the Lemma \ref{lemmaI},
then we get either $R_{S+D}^{+}(\lambda)\leq1$ or $R_{S+D}^{-}(\lambda)\leq1.$
But the equality is excluded by spectrum compactness. So we have at
least 
\begin{equation}
R_{S+D}^{+}(\lambda)<1\label{ineq1}
\end{equation}
 or 
\begin{equation}
R_{S+D}^{-}(\lambda)<1.\label{ineq2}
\end{equation}
If $S$ is invertible, then from (\ref{onlyone}), only one of inequality
(\ref{ineq1}) and (\ref{ineq2}) can be satisfied. Thus, 
\[
\left\{ \lambda\in\mathbb{C}\,:\,R_{S+D}^{+}(\lambda)\geq1\,\mbox{\,\,\ and}\,\,\,R_{S+D}^{-}(\lambda)\geq1\right\} \subset\sigma(T).
\]
Therefore, if $S$ is non invertible, $\sup_{i\in\mathbb{Z}}\left|\dfrac{\prod\limits _{m=1}^{k-1}(d_{i-m}-\lambda)}{\prod\limits _{m=1}^{k}w_{i-m}}\right|$
is not bounded and we have only $R_{S+D}^{+}(\lambda)<1$. So,
\[
\left\{ \lambda\in\mathbb{C}\,:\,R_{S+D}^{+}(\lambda)\geq1\right\} \subset\sigma(T).
\]
Conversely, in order to show that 
\[
\sigma(T)\subset\left\{ \lambda\in\mathbb{C}\,:\,R_{S+D}^{+}(\lambda)\geq1\,\mbox{and}\,R_{1}^{-}(\lambda)\geq1\right\} ,
\]
we take $\lambda\in\mathbb{C}$, such that 
\begin{equation}
R_{S+D}^{+}(\lambda)<1\label{7*}
\end{equation}
or 
\begin{equation}
R_{S+D}^{-}(\lambda)<1\label{8*}
\end{equation}
and we show that $\lambda\notin\sigma(T)$. From Lemma \ref{reciproque},
$T-\lambda I$ is invertible and there will exist an operator $(T-\lambda I)^{-1}\in\mathcal{L}(X)$
such that 
\begin{equation}
I=(T-\lambda I)^{-1}(T-\lambda I)=(T-\lambda I)(T-\lambda I)^{-1}.\label{9*}
\end{equation}
Therefore, $\lambda\notin\sigma(T)$. Similarly one can show that
if $S$ is not invertible, then 
\[
\sigma(T)\subset\left\{ \lambda\in\mathbb{C}\,:\,R_{S+D}^{+}(\lambda)\geq1\right\} .
\]
\end{proof}
\begin{rem}
In the previous theorem, if we take $d_{i}=0$ for all $i\in\mathbb{Z}$,
then we obtain a result already shown in \cite{bourhim2006spectrum,conway2013course,ridge1970approximate,shields1974weighted}
about the spectrum of the operator $S$. That is 
\[
\sigma(S)=\left\{ \lambda\in\mathbb{C}\,:\,\dfrac{1}{r(S^{-1})}\leq\left|\lambda\right|\leq r(S)\right\} .
\]
\end{rem}

\section{Remark about the spectrum of perturbed weighted $n$-shift}

For a strictly positive integer $n$, we define the weighted $n$-shift
operator in $X$ by 
\[
S_{n}e_{i}=w_{i}e_{i+n},\,\,\,i=0,\pm1,\pm2,...
\]
The sequence $\left\{ w_{i}\right\} _{i\in\mathbb{Z}}\subset\mathbb{C}$
represents the weights of the operator $S_{n}$. It is clear that
the weighted $1$-shift coincide with weighted shift (in the usual
sense, see \cite{stochel1989unbounded}).
\begin{rem}
Let $T_{n}\in\mathcal{L}(X)$ be the operator given by
\begin{equation}
T_{n}=S_{n}+D,\label{pertubed_n-1}
\end{equation}
where $D$ is a diagonal operator defined in (\ref{6*}). For every
$j\in\left\{ 0,\,...,\,n-1\right\} $ and $i\in\mathbb{Z}$, let that
$e_{i}^{j}:=e_{j+in}$, $w_{i}^{j}:=w_{j+in}$ and $S_{n}^{j}e_{i}^{j}=w_{i}^{j}e_{i+1}^{j}.$
Where $S_{n}^{j}$ is the restriction of $S_{n}$ on $X_{j}$, the
$S_{n}-$invariant closed linear subspace spanned by $\left\{ e_{i}^{j}\,:\,i\in\mathbb{Z}\right\} $.
Note that 
\[
X=X_{0}\oplus X_{1}\oplus\cdots\oplus X_{n-1}
\]
and 
\[
S_{n}=S_{n}^{0}\oplus S_{n}^{1}\oplus\cdots\oplus S_{n}^{n-1}
\]
Also, since each $S_{n}^{j}$ is a weighted $1$-shift, then the spectra
of $S_{n}$ is the union of the spectra of all $S_{n}^{j}$, $j=0,\,...,\,n-1$
(see \cite{halmos2012hilbert}). In particular, 
\[
\sigma(S_{n})=\sigma(S_{n}^{0})\cup\sigma(S_{n}^{1})\cup\cdot\cdot\cdot\cup\sigma(S_{n}^{n-1}).
\]
Moreover, if we denote by $D^{j}$ the restriction of $D$ to $X_{j}$,
then
\[
S_{n}+D=\left(S_{n}^{0}+D^{0}\right)\oplus\left(S_{n}^{1}+D^{1}\right)\oplus\cdots\oplus\left(S_{n}^{n-1}+D^{n-1}\right)
\]
and thus
\[
\sigma(S_{n}+D)=\sigma(S_{n}^{0}+D^{0})\cup\sigma(S_{n}^{1}+D^{1})\cup\cdots\cup\sigma(S_{n}^{n-1}+D^{n-1})
\]
 Furthermore, for $j\in\left\{ 0,\,...,\,n-1\right\} $, $R_{j}^{+}(\lambda)$
and $R_{j}^{-}(\lambda)$ are given by 
\begin{equation}
R_{j}^{+}(\lambda)=\lim_{k\rightarrow\infty}\left[\sup_{i\in\mathbb{Z}}\left|\dfrac{\prod\limits _{m=0}^{k-1}w_{j+(i+m)n}}{\prod\limits _{m=0}^{k}(d_{j+(i+m)n}-\lambda)}\right|\right]^{\frac{1}{k}}\label{Gamma-1}
\end{equation}
and 
\begin{equation}
R_{j}^{-}(\lambda)=\lim_{k\rightarrow\infty}\left[\sup_{i\in\mathbb{Z}}\left|\dfrac{\prod\limits _{m=0}^{k-1}(d_{j+(i-m)n}-\lambda)}{\prod\limits _{m=0}^{k}w_{j+(i-m)n}}\right|\right]^{\frac{1}{k}}.\label{Delta-1}
\end{equation}
By Theorem \ref{thmprincipal}, if $S_{n}^{j}$ is invertible operator,
then we have 
\[
\sigma(S_{n}^{j}+D^{j})=\left\{ \lambda\in\mathbb{C}\,:\,R_{j}^{+}(\lambda)\geq1\,\,\,\mbox{and}\,\,\,R_{j}^{-}(\lambda)\geq1\right\} ,
\]
and if $S_{n}^{j}$ is non-invertible operator then 
\[
\sigma(S_{n}^{j}+D^{j})=\left\{ \lambda\in\mathbb{C}\,:\,R_{j}^{+}(\lambda)\geq1\right\} .
\]
\end{rem}
%%%%%%%%%%%%
\bibliographystyle{abbrv}
%\bibliography{bib_2d_map}

\end{document}